\documentclass[a4paper,12pt]{article}
\usepackage{amssymb}
\usepackage{amsthm}
\usepackage[dvips]{graphicx}
\usepackage{amsmath}

\theoremstyle{plain}
\newtheorem{theorem}{Theorem}[section]
\newtheorem{proposition}{Proposition}[section]
\newtheorem{lemma}{Lemma}[section]

\newtheorem{remark}{\sc Remark}[section]
\newtheorem{example}{\sc Example}[section]

\renewcommand{\baselinestretch}{1.2}

\newenvironment{reference}[1]{%

\begin{flushleft}\normalsize{\textbf{References}}\end{flushleft}%
\begin{enumerate}\setlength{\itemsep}{-5pt}\small
}{\end{enumerate}}

\renewcommand{\section}[1]{%
\addtocounter{section}{1}\setcounter{subsection}{0}%
\begin{flushleft}\vspace*{10pt}
\normalsize{\textbf{\arabic{section}. #1}}
\end{flushleft}}
\renewcommand{\subsection}[1]{\addtocounter{subsection}{1}%
\noindent\textbf{\arabic{section}.\arabic{subsection}. #1 }}

\title{\vspace*{-21mm}
\large{\textbf{
\ \ On the $\mathbb Z_p$-ranks of tamely ramified Iwasawa modules
}}
\footnotetext{
Postprint of an article published in {\em  Int.\ J.\ Number Theory}, DOI: 10.1142/S1793042113500395 \copyright World Scientific Publishing Company \ \texttt{http://www.worldscientific.com/worldscinet/ijnt}
}
}

\author{ 
\textsc{\normalsize Tsuyoshi Itoh}
\and 
\textsc{\normalsize Yasushi Mizusawa}
\and 
\textsc{\normalsize Manabu Ozaki}
}

\date{}

\begin{document}
{\renewcommand{\baselinestretch}{1.05} \maketitle}

\vspace*{-13mm}
\renewcommand{\abstractname}{}
{\renewcommand{\baselinestretch}{1.05}
\begin{abstract}{\footnotesize 
\noindent\textbf{Abstract.} 
For a finite set $S$ of prime numbers, we consider the $S$-ramified Iwasawa module which is the Galois group of the maximal abelian pro-$p$-extension unramified outside $S$ over the cyclotomic $\mathbb Z_p$-extension of a number field $k$. In the case where $S$ does not contain $p$ and $k$ is the rational number field or an imaginary quadratic field, we give the explicit formulae of the $\mathbb Z_p$-ranks of the $S$-ramified Iwasawa modules by using Brumer's $p$-adic version of Baker's theorem on the linear independence of logarithms of algebraic numbers. 
}\end{abstract}}

\thispagestyle{empty}

%------------------------------------------------------------------------------
\section{Introduction}

Let $p$ be a prime number, and $\mathbb Q_{\infty}$ the $\mathbb Z_p$-extension of the rational number field $\mathbb Q$. For a finite extension $k/\mathbb Q$, the composition $k_{\infty}=k \mathbb Q_{\infty}$ is the cyclotomic $\mathbb Z_p$-extension of the number field $k$. The Galois group $\varGamma=\mathrm{Gal}(k_{\infty}/k)$ is isomorphic to the additive group of the ring $\mathbb Z_p$ of $p$-adic integers. For a finite set $S$ of prime numbers, the $S$-ramified Iwasawa module $X_S(k_{\infty})=\mathrm{Gal}(L_S(k_{\infty})/k_{\infty})$ is defined as the Galois group of the maximal abelian pro-$p$-extension $L_S(k_{\infty})/k_{\infty}$ unramified outside $S$ (and unramified at infinite places), on which $\varGamma$ acts via the inner automorphisms of $\mathrm{Gal}(L_S(k_{\infty})/k)$. Then the complete group ring $\mathbb Z_p[[\varGamma]]$ acts on $X_S(k_{\infty})$ continuously, and $X_S(k_{\infty})$ is finitely generated as a $\mathbb Z_p[[\varGamma]]$-module. A description of the explicit structure of $X_S(k_{\infty})$ is one of the main themes in Iwasawa theory. 

In the case that $p \in S$, the structure of $X_S(k_{\infty})$ is relatively well known. For instance, $X_S(k_{\infty})$ is a finitely generated free $\mathbb Z_p$-module if $k=\mathbb Q$ and $p \in S$ (cf.\ \cite{NSW2} Theorem 10.5.6, \cite{Iwa81} Theorem 3 etc.). On the other hand, particularly when $S=\emptyset$, the structural invariants of the unramified Iwasawa module $X_{\emptyset}(k_{\infty})$ appear in Iwasawa's class number formula \cite{Iwa59} (cf.\ \cite{Was} etc.). Iwasawa \cite{Iwa73} conjectured that $X_{\emptyset}(k_{\infty})$ is a finitely generated $\mathbb Z_p$-module, and Greenberg \cite{Gre76} conjectured that $X_{\emptyset}(k_{\infty})$ is finite if $k$ is a totally real number field. These conjectures are unsolved in general at present, while many affirmative results are known. In the case that $p \not\in S$, the tamely ramified Iwasawa module $X_S(k_{\infty})$ is still an uncertain object, even if the unramified quotient $X_{\emptyset}(k_{\infty})$ is well understood. 

In this paper, we focus on the $\mathbb Z_p$-rank 
\[
\lambda_S(k_{\infty}) = \mathrm{rank}_{\mathbb Z_p} X_S(k_{\infty}) = \dim_{\mathbb Q_p} (X_S(k_{\infty}) \otimes_{\mathbb Z_p} \mathbb Q_p) 
\]
of the tamely ramified Iwasawa module $X_S(k_{\infty})$, where $\mathbb Q_p$ denotes the field of $p$-adic numbers. The relative formulae of $\lambda_{\emptyset}(k_{\infty})$ for the $p$-extensions of $J$-fields are well known as the formulae of Riemann-Hurwitz type by Iwasawa \cite{Iwa81} and Kida \cite{Kida80} \cite{Kida82}. In the case where $p=2$ and $k$ is an imaginary quadratic field, Ferrero \cite{Fer80} and Kida \cite{Kida79} gave the explicit formulae of $\lambda_{\emptyset}(k_{\infty})$, and Salle \cite{Salle} gave a relative formula between $\lambda_S(k_{\infty})$ and $\lambda_S(\mathbb Q_{\infty})$ with its applications in some special cases where $\lambda_S(\mathbb Q_{\infty})=0$. As a generalization of these formulae, we give the formulae of $\lambda_S(k_{\infty})$ in the most basic case where $k=\mathbb Q$ and also for imaginary quadratic fields $k$. 

For a prime number $\ell \neq p$, we denote by $P_{\ell}=p^{N_{\ell}}$ the number of places of $\mathbb Q_{\infty}$ lying over $\ell$. If $2 \neq \ell \equiv \pm 1 \pmod{p}$, $N_{\ell}$ is the largest integer such that $\ell^2 \equiv 1 \pmod{4p^{N_{\ell}+1}}$. If $\ell=2$ and $p=3$, then $P_{\ell}=1$. We regard ``$\max \emptyset =0$". The main results are the following theorems. 

\begin{theorem}\label{mainthm}
Let $p$ be a prime number, and $S$ a finite set of prime numbers such that $p \not\in S$. 

{\em (i)} If $p \neq 2$, put $S^{\prime}=\{\,\ell \in S\ |\ \ell \equiv 1\pmod{p}\,\}$ and $P_{\max}^{\prime}=\max\{\, P_{\ell} \ |\ \ell \in S^{\prime} \,\}$. Then 
\[
\lambda_S(\mathbb Q_{\infty}) = \sum_{\ell \in S^{\prime}} P_{\ell} - P_{\max}^{\prime} . 
\]

{\em (ii)} If $p=2$, put $S^{\pm}=\{\, \ell \in S \ |\ \ell \equiv \pm 1 \pmod{4}\,\}$, $P_{\max}^+ = \max\{\, P_{\ell} \ |\ \ell \in S^+\,\}$ and $\mathbb P=\{\, P_{\ell} \ |\ \ell \in S^-,\ P_{\ell} \ge P_{\max}^+\,\}$. Then 
\[
\lambda_S(\mathbb Q_{\infty}) = \sum_{\ell \in S} P_{\ell} - P_{\max}^+ - \sum_{P \in \mathbb P} P. 
\]
\end{theorem}

\begin{remark}{\em 
For an arbitrary integer $\lambda \ge 0$, one can find $S$ such that $\lambda_S(\mathbb Q_{\infty})=\lambda$. Refer to \cite{FOO06} for number fields $k$ with prescribed $\lambda_{\emptyset}(k_{\infty})$. 
}\end{remark}

\begin{example}{\em 
If $p=2$ and $S=\{3,5,7,31,79\}$, then $\lambda_S(\mathbb Q_{\infty})=0$. On the other hand, the maximal pro-$2$-extension over $\mathbb Q$ unramified outside $S$ is an infinite extension (cf.\ \cite{NSW2} Corollary 10.10.2). This example is analogous to the situation of \cite{Miz00} and \cite{Oza09} Corollary 1. 
}\end{example}

\begin{theorem}\label{imagquad}
Let $p$ be a prime number, and $S$ a finite set of prime numbers such that $p \not\in S$. Let $k$ be an imaginary quadratic field. 

{\em (i)} If $p \neq 2$, put $S^{\prime}$ and $P_{\max}^{\prime}$ as in Theorem \ref{mainthm}, and put $S^{\prime\prime}=\{\,\ell \in S\ |\ \ell \equiv -1\pmod{p},\ \hbox{\rm $\ell$ is inert in $k$}\,\}$. Assume that $S^{\prime} \cup S^{\prime\prime} \neq \emptyset$. Then 
\[
\lambda_S(k_{\infty}) 
= \lambda_{\emptyset}(k_{\infty}) 
+ \sum_{\ell \in S^{\prime} \cup S^{\prime\prime}} P_{\ell} 
+ \sum_{\substack{\ell \in S^{\prime} \\ {\rm split\,in\,} k}} P_{\ell} 
- P_{\max}^{\prime} 
- \delta 
\]
where $\delta=1$ if $p=3$ and $k=\mathbb Q(\sqrt{-3})$, and $\delta=0$ otherwise. 

{\em (ii)} If $p=2$, put $P_{\max}^+$ and $\mathbb P$ as in Theorem \ref{mainthm}, and let $S_k$ be the set of odd prime numbers ramifying in $k$. Assume that $S \cup S_k \neq \emptyset$. Then 
\[
\lambda_S(k_{\infty}) 
= 2\sum_{\ell \in S} P_{\ell} 
+ \sum_{\ell \in S_k \setminus S} P_{\ell}
- 1
- P_{\max}^+ 
- \sum_{P \in \mathbb P} P .
\]
\end{theorem}

\begin{remark}{\em 
If $p=3$ and $k=\mathbb Q(\sqrt{-3})$, then $\lambda_{\emptyset}(k_{\infty})=0$ and any $\ell \in S^{\prime}$ splits in $k$. The invariants $\lambda_{\emptyset}(k_{\infty})$ are also computable via the Stickelberger elements (cf., e.g., \cite{IS96}, \cite{Was}). 
}\end{remark}

In Section 2, we recall some basic facts, and reduce Theorems \ref{mainthm} and \ref{imagquad} to the finiteness of $X_{\{\ell\}}(\mathbb Q_{\infty})$ (Lemma \ref{reduce}). In Section 3, we prove the finiteness of $X_{\{\ell\}}(\mathbb Q_{\infty})$ as a key theorem (Theorem \ref{keythm}) by using Brumer's $p$-adic version \cite{Brumer} of Baker's theorem (cf.\,\cite{Was} Theorem 5.29). Then we also give an application of the key theorem to finite fields (Proposition \ref{finitefield}). Finally, we give an upper bound of $|X_{\{\ell\}}(\mathbb Q_{\infty})|$ in a special case where $p=3$. 

%------------------------------------------------------------------------------
\section{Preliminaries}

Put $\zeta_{n}=\exp(2\pi\sqrt{-1}/n)$ for each $1 \le n \in \mathbb Z$. Put $\mathbb B_n=\mathbb Q_{\infty} \cap \mathbb Q(\zeta_{2p^{n+1}})$ which is the $n$-th layer of the $\mathbb Z_p$-extension $\mathbb Q_{\infty}/\mathbb Q$, i.e., the unique subextension of degree $p^n$. For a finite extension $k/\mathbb Q$ such that $k \cap \mathbb Q_{\infty} = \mathbb Q$, $k_n=k \mathbb B_n$ is the $n$-th layer of the cyclotomic $\mathbb Z_p$-extension $k_{\infty}/k$. Then we identify $\mathrm{Gal}(k_{\infty}/k)$ with $\varGamma=\mathrm{Gal}(\mathbb Q_{\infty}/\mathbb Q)$ via the restriction mapping. For the cyclotomic $\mathbb Z_p$-extension $\mathbb Q(\mu_{p^{\infty}})=\bigcup_{n \ge 1}\mathbb Q(\zeta_{p^n})$ of $\mathbb Q(\zeta_{2p})$, we take the generator $\overline{\gamma}$ of $\mathrm{Gal}(\mathbb Q(\mu_{p^{\infty}})/\mathbb Q(\zeta_{2p}))$ such that $\zeta_{p^n}^{\overline{\gamma}}=\zeta_{p^n}^{1+q}$ for all $n$, where $q=p$ if $p \neq 2$, and $q=4$ if $p=2$. Then $\gamma=\overline{\gamma}|_{\mathbb Q_{\infty}}$ is a generator of $\varGamma=\mathrm{Gal}(\mathbb Q_{\infty}/\mathbb Q)$. The fixed field of $\varGamma^{p^n}$ is the $n$-th layer of the $\mathbb Z_p$-extension. By identifying $\gamma$ with $1+T$, we can regard a $\mathbb Z_p[[\varGamma]]$-module as a module over the ring $\varLambda=\mathbb Z_p[[T]]$ of formal power series. 

Let $k/\mathbb Q$ be a finite extension such that $k \cap \mathbb Q_{\infty} = \mathbb Q$. Let $O_{k_n}$ be the ring of algebraic integers in $k_n$, and $E(k_n)=O_{k_n}^{\times}$ the unit group. For a finite set $S$ of prime numbers, $L_S(k_n)$ denotes the maximal abelian pro-$p$-extension of $k_n$ unramified outside $S$. If $p \not\in S$, we denote by $A_S(k_n)$ the Sylow $p$-subgroup of the ray class group of $k_n$ modulo $\mathfrak{m}=\prod_{\ell \in S} \ell$. In particular, $A_{\emptyset}(k_n)$ is the Sylow $p$-subgroup of the ideal class group of $k_n$. Then we obtain the exact sequence 
\begin{equation}\tag{\thesection.1}\label{ray0}
E(k_n) \otimes_{\mathbb Z} \mathbb Z_p \rightarrow (O_{k_n}/\mathfrak{m})^{\times} \otimes_{\mathbb Z} \mathbb Z_p \rightarrow A_S(k_n) \rightarrow A_{\emptyset}(k_n) \rightarrow 0
\end{equation}
for each $n$. Since $\mathrm{Gal}(L_S(k_n)/k_n) \simeq A_S(k_n)$ by class field theory, we also obtain the exact sequence 
\begin{equation}\tag{\thesection.2}\label{ray}
\varprojlim E(k_n) \otimes_{\mathbb Z} \mathbb Z_p \rightarrow \varprojlim (O_{k_n}/\mathfrak{m})^{\times} \otimes_{\mathbb Z} \mathbb Z_p \rightarrow X_S(k_{\infty}) \rightarrow X_{\emptyset}(k_{\infty}) \rightarrow 0
\end{equation}
of Galois modules by taking the projective limit $\varprojlim$ of the sequence (\ref{ray0}) with respect to the norm mappings. Since no prime ideals of $k$ split completely in the cyclotomic $\mathbb Z_p$-extension $k_{\infty}/k$, $\varprojlim (O_{k_n}/\mathfrak{m})^{\times} \otimes_{\mathbb Z} \mathbb Z_p$ is finitely generated as a $\mathbb Z_p$-module. Therefore, by the theorem of Ferrero and Washington \cite{FW79} (cf.\ \cite{Was} Theorem 7.15, Proposition 13.23), the tamely ramified Iwasawa module $X_S(k_{\infty})$ is finitely generated as a $\mathbb Z_p$-module if $k/\mathbb Q$ is abelian. 

For prime numbers $\ell \neq p$, we define the polynomials $f_{\ell}(T) \in \varLambda$ as follows: 
\[
f_{\ell}(T)=
\left\{\begin{array}{ll}
(1+T)^{P_{\ell}}-(1+p)^{P_{\ell}} & : p \neq 2,\ \ell \equiv 1\ (\mathrm{mod}\,p), \\
(1+T)^{P_{\ell}}-(-1)^{\frac{\ell-1}{2}} 5^{P_{\ell}} & : p=2, \\
1 & : \ell \not\equiv 1\ (\mathrm{mod}\,p). 
\end{array}\right.
\]

\begin{lemma}\label{limO}
For $\ell \neq p$, $\varprojlim (O_{\mathbb B_n}/\ell)^{\times} \otimes_{\mathbb Z} \mathbb Z_p \simeq \varLambda/f_{\ell}(T)$ as a $\varLambda$-module. 
\end{lemma}

\begin{proof}
The decomposition field of $\ell$ in $\mathbb Q_{\infty}/\mathbb Q$ is $\mathbb B_{N_{\ell}}$. Since $|(O_{\mathbb B_n}/\ell)^{\times}|=(\ell^{p^{n-N_{\ell}}}-1)^{P_{\ell}} \equiv \ell-1 \pmod{p}$, the claim holds if $\ell \not\equiv 1 \pmod{p}$. 

Suppose that $\ell \equiv 1 \pmod{p}$. Let $F$ be the decomposition field of $\ell$ in $\mathbb Q(\mu_{p^{\infty}})/\mathbb Q$. Then $F=\mathbb Q(\zeta_{2^{N_{\ell}+3}}-\zeta_{2^{N_{\ell}+3}}^{-1})$ if $p=2$ and $\ell \equiv -1 \pmod{4}$, and $F=\mathbb Q(\zeta_{qp^{N_{\ell}}})$ otherwise. Let $\mathfrak{L}$ be a prime ideal of $F$ lying over $\ell$, and put $\mathfrak{l}=\mathfrak{L} \cap O_{\mathbb B_{N_{\ell}}}$. Then $(O_{\mathbb B_n}/\ell)^{\times} \simeq \prod_{i=0}^{P_{\ell}-1} (O_{\mathbb B_n}/\mathfrak{l}^{\gamma^i})^{\times}$ for $n \ge N_{\ell}$. 
Let $\eta_0$ be a generator of $\varprojlim (O_{\mathbb B_n}/\mathfrak{l})^{\times} \otimes_{\mathbb Z} \mathbb Z_p \simeq \mathbb Z_p$. Then $\varprojlim (O_{\mathbb B_n}/\ell)^{\times} \otimes_{\mathbb Z} \mathbb Z_p \simeq \prod_{i=0}^{P_{\ell}-1} \varprojlim (O_{\mathbb B_n}/\mathfrak{l}^{\gamma^i})^{\times} \otimes_{\mathbb Z} \mathbb Z_p$ is a cyclic $\varLambda$-module generated by $\eta \leftrightarrow (\eta_0,1,\cdots,1)$. Let $\mathfrak{L}^{\gamma^i}$ be a prime ideal of $F$ lying over $\mathfrak{l}^{\gamma^i}$. Then $O_{\mathbb B_{N_{\ell}}}/\mathfrak{l}^{\gamma^i} \simeq O_F/\mathfrak{L}^{\gamma^i} \simeq \mathbb F_{\ell}$ and 
\[
(O_{\mathbb B_n}/\mathfrak{l}^{\gamma^i})^{\times} \otimes_{\mathbb Z} \mathbb Z_p \simeq (\mathbb Z[\zeta_{qp^n}]/\mathfrak{L}^{\gamma^i})^{\times} \otimes_{\mathbb Z} \mathbb Z_p \simeq \langle \zeta_{qp^n}\,\mathrm{mod}\,\mathfrak{L}^{\gamma^i} \rangle \otimes_{\mathbb Z} \mathbb Z_p \simeq \langle \zeta_{qp^n} \rangle
\]
for all $n>N_{\ell}$ as $\mathrm{Gal}(\mathbb Q(\mu_{p^{\infty}})/F)$-modules. Then $\gamma_F \in \mathrm{Gal}(\mathbb Q(\mu_{p^{\infty}})/F) \simeq \varGamma^{P_{\ell}}$ such that $\gamma_F|_{\mathbb Q_{\infty}}=\gamma^{P_{\ell}}$ satisfies $\zeta_{p^{n+1}}^{\gamma_F}=\zeta_{p^{n+1}}^{(1+p)^{P_{\ell}}}$ if $p \neq 2$, and $\zeta_{2^{n+2}}^{\gamma_F}=\zeta_{2^{n+2}}^{\pm 5^{P_{\ell}}}$ if $p=2$ and $\ell \equiv \pm 1 \pmod{4}$ respectively. Therefore $f_{\ell}(T)$ annihilates $\varprojlim (O_{\mathbb B_n}/\ell)^{\times} \otimes_{\mathbb Z} \mathbb Z_p$. Then the homomorphism $\varLambda/f_{\ell}(T) \rightarrow \varprojlim (O_{\mathbb B_n}/\ell)^{\times} \otimes_{\mathbb Z} \mathbb Z_p : 1\,\mathrm{mod}\,f_{\ell}(T) \mapsto \eta$ is an isomorphism, since the both sides are isomorphic to ${\mathbb Z_p}^{\!\oplus P_{\ell}}$ as $\mathbb Z_p$-modules. 
\end{proof}

For the convenience, we give the proof of the following lemma which is well known. 

\begin{lemma}\label{limE}
$\mathcal{E}=\varprojlim E(\mathbb B_n) \otimes_{\mathbb Z} \mathbb Z_p \simeq \varLambda$ as a $\varLambda$-module. 
\end{lemma}

\begin{proof}
For an abelian extension $k/\mathbb B_n$, we denote by $N_{k/\mathbb B_n}$ the norm mapping. Put 
\[
\xi_n=\bigg\{
\begin{array}{ll}
N_{\mathbb Q(\zeta_{p^{n+1}})/\mathbb B_n} 
\big( \zeta_{p^{n+1}}^{(1-(1+p)\omega)/2} (1-\zeta_{p^{n+1}}^{(1+p)\omega})/(1-\zeta_{p^{n+1}}) \big) & : p \neq 2, \\
\zeta_{2^{n+2}}^{-2} (1-\zeta_{2^{n+2}}^5)/(1-\zeta_{2^{n+2}}) & : p=2
\end{array}
\]
where $\omega$ is a primitive ($p-1$)-th root of unity in $\mathbb Z_p$. Then $\xi_n \in E(\mathbb B_n)$ and $N_{\mathbb B_{n+1}/\mathbb B_n} \xi_{n+1}=\xi_n$ for all $n \ge 0$. For a moment, we suppose that $p \neq 2$. Let $C(\mathbb Q(\zeta_{p^{n+1}}))$ (resp.\ $C(\mathbb B_n)$) be the group of circular units of $\mathbb Q(\zeta_{p^{n+1}})$ (resp.\ $\mathbb B_n$) defined by Sinnott \cite{Sin80}. Then $\langle \zeta_{p^{n+1}} \rangle C(\mathbb Q(\zeta_{p^{n+1}}))$ is the group of cyclotomic units of $\mathbb Q(\zeta_{p^{n+1}})$ defined in \cite{Was}. Since $N_{\mathbb Q(\zeta_{p^{n+1}})/\mathbb B_n} \zeta_{p^{n+1}}=1$, Lemma 8.1 (b) and Proposition 8.11 of \cite{Was} yield that $N_{\mathbb Q(\zeta_{p^{n+1}})/\mathbb B_n} C(\mathbb Q(\zeta_{p^{n+1}}))$ is generated by one element $\xi_n$ as a $\mathbb Z[\varGamma/\varGamma^{p^n}]$-module. Since $C(\mathbb B_n)^{p-1} \subset N_{\mathbb Q(\zeta_{p^{n+1}})/\mathbb B_n} C(\mathbb Q(\zeta_{p^{n+1}})) \subset C(\mathbb B_n)$ and $|E(\mathbb B_n)/C(\mathbb B_n)|$ is prime to $p$ by Theorems 4.1 and 5.1 of \cite{Sin80}, $E(\mathbb B_n) \otimes_{\mathbb Z}\mathbb Z_p$ is generated by $\xi_n \otimes 1$ as a $\mathbb Z_p[\varGamma/\varGamma^{p^n}]$-module. If $p=2$, we have $N_{\mathbb B_n/\mathbb Q} \xi_n = -1$. Then Theorem 8.2 and Proposition 8.11 of \cite{Was} yield that $\xi_n \otimes 1$ generates $E(\mathbb B_n) \otimes_{\mathbb Z}\mathbb Z_2$ as a $\mathbb Z_2[\varGamma/\varGamma^{2^n}]$-module. Therefore $\mathcal{E}$ is a cyclic $\varLambda$-module. Since the $\mathbb Z_p$-rank of $E(\mathbb B_n) \otimes_{\mathbb Z}\mathbb Z_p$ is unbounded as $n \rightarrow \infty$, we have $\mathcal{E} \simeq \varLambda$. 
\end{proof}

The main theorems are reduced to the finiteness of $X_{\{\ell\}}(\mathbb Q_{\infty})$ as follows. 

\begin{lemma}\label{reduce}
If $X_{\{\ell\}}(\mathbb Q_{\infty})$ is finite for any $\ell \neq p$, then both Theorems \ref{mainthm} and \ref{imagquad} hold. 
\end{lemma}

\begin{proof}
By Lemma \ref{limE}, $\mathcal{E} \simeq \varLambda$. By Lemma \ref{limO} and the sequence (\ref{ray}), we obtain the exact sequence 
\[
\mathcal{E}/f_{\ell}(T) \stackrel{\eta_{\ell}}{\longrightarrow} \varprojlim (O_{\mathbb B_n}/\ell)^{\times} \otimes_{\mathbb Z} \mathbb Z_p \rightarrow X_{\{\ell\}}(\mathbb Q_{\infty}) 
\]
for each $\ell \neq p$. Since the free $\mathbb Z_p$-modules $\mathcal{E}/f_{\ell}(T)$ and $\varprojlim (O_{\mathbb B_n}/\ell)^{\times} \otimes_{\mathbb Z} \mathbb Z_p$ have the same rank, the finiteness of $X_{\{\ell\}}(\mathbb Q_{\infty})$ yields the injectivity of $\eta_{\ell}$. Then we obtain the injective homomorphisms
\[
\mathcal{E}/f(T) \stackrel{\xi}{\longrightarrow} \bigoplus_{\ell \in S} \mathcal{E}/f_{\ell}(T) \stackrel{\oplus \eta_{\ell}}{\longrightarrow} \bigoplus_{\ell \in S} \varprojlim (O_{\mathbb B_n}/\ell)^{\times} \otimes_{\mathbb Z} \mathbb Z_p \stackrel{\oplus \iota_{\ell}}{\longrightarrow} \bigoplus_{\ell \in S} \varprojlim (O_{k_n}/\ell)^{\times} \otimes_{\mathbb Z} \mathbb Z_p
\]
where $\xi$ is the diagonal mapping, $f(T)=\mathrm{lcm}\{\,f_{\ell}(T)\ |\ \ell \in S\,\}$, and $\iota_{\ell}$ is the natural mapping induced from the inclusions $\mathbb B_n \subset k_n$. 

Since $X_{\emptyset}(\mathbb Q_{\infty})=0$, the sequence (\ref{ray}) induces the exact sequence
\[
0 \rightarrow \mathcal{E}/f(T) \stackrel{\oplus \eta_{\ell} \circ \xi}{\longrightarrow} \bigoplus_{\ell \in S} \varprojlim (O_{\mathbb B_n}/\ell)^{\times} \otimes_{\mathbb Z} \mathbb Z_p \rightarrow X_S(\mathbb Q_{\infty}) \rightarrow 0 . 
\]
By Lemma \ref{limO}, we have $\lambda_S(\mathbb Q_{\infty}) = \sum_{\ell \in S} \deg f_{\ell}(T) - \deg f(T)$. If $p \neq 2$, then $\sum_{\ell \in S} \deg f_{\ell}(T)=\sum_{\ell \in S^{\prime}} P_{\ell}$ and $f(T)=(1+T)^{P_{\max}^{\prime}}-(1+p)^{P_{\max}^{\prime}}$. If $p=2$, then $\deg f_{\ell}(T)=P_{\ell}$ for any $\ell \in S$, and 
\[
f(T) = ((1+T)^{P_{\max}^+} - 5^{P_{\max}^+}) \prod_{P \in \mathbb P} ((1+T)^{P} + 5^{P}) . 
\]
Thus we obtain the claims of Theorem \ref{mainthm}. 

Let $Q_{\ell}$ be the number of places of $k_{\infty}$ lying over $\ell$. Put $S^{\prime}=\{\,\ell \in S\ |\ \ell \equiv 1\pmod{p}\,\}$ and $S^{\prime\prime}=\{\,\ell \in S\ |\ \ell \equiv -1\pmod{p},\ \hbox{\rm $\ell$ is inert in $k$}\,\}$ even if $p=2$. Then $\varprojlim (O_{k_n}/\ell)^{\times} \otimes_{\mathbb Z} \mathbb Z_p$ is a free $\mathbb Z_p$-module of rank $Q_{\ell}$ or $0$ if $\ell \in S^{\prime} \cup S^{\prime\prime}$ or not respectively. Let $J$ be a generator of $\mathrm{Gal}(k_{\infty}/\mathbb Q_{\infty})$ which acts as the complex conjugation. Put $W=\{1\}$ if $\zeta_{q} \not\in k_{\infty}$, and put $W=\varprojlim \langle \mu_{qp^n} \rangle \simeq \mathbb Z_p$ if $\zeta_q \in k_{\infty}$. Then the natural $\varLambda[J]$-homomorphism 
\[
\xi^- : W \rightarrow \bigoplus_{\ell \in S} \varprojlim (O_{k_n}/\ell)^{\times} \otimes_{\mathbb Z} \mathbb Z_p
\]
is also injective. Note that $\varprojlim E(k_n) \otimes_{\mathbb Z} \mathbb Z_p \simeq W \oplus \mathcal{E}$ as $\varLambda[J]$-modules. Recall that $\xi^+ = (\oplus \iota_{\ell}) \circ (\oplus \eta_{\ell}) \circ \xi$ is an injective $\varLambda[J]$-homomorphism. By the exact sequence (\ref{ray}), we obtain the following commutative diagram with exact rows. 
\[
\begin{array}{cccl}
\varprojlim E(k_n) \otimes_{\mathbb Z} \mathbb Z_p &\longrightarrow& \displaystyle{\bigoplus_{\ell \in S} \varprojlim (O_{k_n}/\ell)^{\times} \otimes_{\mathbb Z} \mathbb Z_p} & \rightarrow X_S(k_{\infty}) \rightarrow X_{\emptyset}(k_{\infty}) \rightarrow 0 \\
\downarrow && \parallel & \\
W \oplus \mathcal{E}/f(T) &\stackrel{\xi^- + \xi^+}{\longrightarrow}& \displaystyle{\bigoplus_{\ell \in S} \varprojlim (O_{k_n}/\ell)^{\times} \otimes_{\mathbb Z} \mathbb Z_p} & 
\end{array}
\]
Since $J$ acts on $\mathrm{Im}\,\xi^{\pm}$ as $\pm 1$ respectively, $2(\mathrm{Im}\,\xi^+ \cap \mathrm{Im}\,\xi^-)=0$ and hence $\mathrm{Im}\,\xi^+ \cap \mathrm{Im}\,\xi^- =0$. Therefore $\xi^- + \xi^+$ is injective, and hence 
\[
\lambda_S(k_{\infty}) = \lambda_{\emptyset}(k_{\infty}) + \sum_{\ell \in S^{\prime} \cup S^{\prime\prime}} Q_{\ell} - \deg f(T) - \delta 
\]
where $\delta=\mathrm{rank}_{\mathbb Z_p}W$. 
Suppose that $p \neq 2$. Then $Q_{\ell}=2P_{\ell}$ if $\ell$ splits in $k/\mathbb Q$, and $Q_{\ell}=P_{\ell}$ otherwise. Thus we obtain the claim (i) of Theorem \ref{imagquad}. 
Suppose that $p=2$. Then $S^{\prime\prime} \subset S^{\prime}=S$. For each $\ell \in S$, $Q_{\ell}=2P_{\ell}$ if $\ell \not\in S_k$, and $Q_{\ell}=P_{\ell}$ otherwise. By the formula of Ferrero \cite{Fer80} and Kida \cite{Kida79}, we have $\lambda_{\emptyset}(k_{\infty})=-1+\sum_{\ell \in S_k} P_{\ell}$ if $\delta=0$. If $\delta=1$, then $\lambda_{\emptyset}(k_{\infty})=0$ and $S_k=\emptyset$. By combining them, we obtain the claim (ii) of Theorem \ref{imagquad}. This completes the proof of Lemma \ref{reduce}. 
\end{proof}

%------------------------------------------------------------------------------
\section{Proofs of theorems}

In this section, we complete the proofs of Theorems \ref{mainthm} and \ref{imagquad}. By Lemma \ref{reduce}, it suffices to show the following theorem which is a special case of Theorem \ref{mainthm}. 

\begin{theorem}\label{keythm}
$X_{\{\ell\}}(\mathbb Q_{\infty})$ is finite for any $\ell \neq p$. 
\end{theorem}

In order to prove this theorem, we prepare some lemmas and notations. 

\begin{lemma}\label{Xpell}
$X_{\{p,\ell\}}(\mathbb Q_{\infty}) \simeq {\mathbb Z_p}^{\!\oplus P_{\ell}}$ as a $\mathbb Z_p$-module if $\ell \equiv 1 \pmod{p}$. 
\end{lemma}

\begin{proof}
For a finite place $v$ of $\mathbb B_n$, we denote by $U_{n,v}$ the local unit group of the completion of $\mathbb B_n$ at $v$. Since Leopoldt's conjecture holds for $\mathbb B_n$, the class field theory provides the following commutative diagram with exact rows (cf.\ arguments of \S 13.1 \cite{Was}). 
\[
\begin{array}{cccccl}
E(\mathbb B_n) \otimes_{\mathbb Z} \mathbb Z_p &\rightarrow& \prod_{v|p,\ell} U_{n,v} \otimes_{\mathbb Z} \mathbb Z_p &\rightarrow& \mathrm{Gal}(L_{\{p,\ell\}}(\mathbb B_n)/\mathbb B_n) &\rightarrow 0 \\
\parallel && \downarrow \hbox{\scriptsize{proj.}} && \downarrow \hbox{\scriptsize{restr.}} & \\
E(\mathbb B_n) \otimes_{\mathbb Z} \mathbb Z_p &\stackrel{\iota_n}{\longrightarrow}& \prod_{v|p} U_{n,v} \otimes_{\mathbb Z} \mathbb Z_p &\rightarrow& \mathrm{Gal}(L_{\{p\}}(\mathbb B_n)/\mathbb B_n) &\rightarrow 0 
\end{array}
\]
Note that $L_{\{p\}}(\mathbb B_n)=\mathbb Q_{\infty}$ and $\varprojlim \mathrm{Gal}(L_{\{p\}}(\mathbb B_n)/\mathbb B_n) =0$. By Theorem (11.2.4) of \cite{NSW2}, $\varprojlim \prod_{v|p} U_{n,v} \otimes_{\mathbb Z} \mathbb Z_p$ is a submodule of $\varLambda$ of finite index, and hence $\varprojlim \iota_n$ is isomorphism. By taking $\varprojlim$ of the above diagram and using snake lemma, we have
\[
X_{\{p,\ell\}}(\mathbb Q_{\infty}) = \mathrm{Gal}(L_{\{p,\ell\}}(\mathbb Q_{\infty})/\mathbb Q_{\infty}) \simeq \hbox{$\varprojlim \prod_{v|\ell} U_{n,v} \otimes_{\mathbb Z} \mathbb Z_p$} \simeq {\mathbb Z_p}^{\!\oplus P_{\ell}} 
\]
as a $\mathbb Z_p$-module. 
\end{proof}

Let $\overline{\mathbb Q}$ (resp.\ $\overline{\mathbb Q}_p$) be the algebraic closure of $\mathbb Q$ (resp.\ $\mathbb Q_p$), and $\mathbb C_p$ the completion of $\overline{\mathbb Q}_p$. Fix an embedding $\overline{\mathbb Q} \hookrightarrow \mathbb C_p$, and let $v_p : \mathbb C_p \rightarrow \mathbb Q \cup \{\infty\}$ be the additive $p$-adic valuation such that $v_p(p)=1$. Let $F$ be the decomposition field of $\ell$ in $\mathbb Q(\mu_{p^{\infty}})/\mathbb Q$. Then $G=\mathrm{Gal}(F/\mathbb Q)$ can be regarded as a quotient of $\mathrm{Gal}(\mathbb Q_p(\mu_{p^{\infty}})/\mathbb Q_p)$. Let $\mathfrak{L}$ be a prime ideal of $F$ lying over $\ell$, and $h$ be the class number of $F$. Then $\mathfrak{L}^{(p-1) h}=\alpha O_F$ is a principal ideal of the ring $O_F$ of algebraic integers in $F$ generated by some $\alpha \in F^{\times}$ such that $v_p(\alpha-1)>0$. 
Put $\varDelta=\mathrm{Gal}(\mathbb Q(\mu_{p^{\infty}})/\mathbb Q_{\infty})$, and let $\omega : \varDelta \rightarrow \mathbb Z_p^{\times}$ be the Teichm\"uller character satisfying $\zeta_{p^n}^{\delta}=\zeta_{p^n}^{\omega(\delta)}$ for all $\delta \in \varDelta$ and $n \ge 1$. Put $e = \frac{1}{p-1} \sum_{\delta \in \varDelta} \omega(\delta) \delta^{-1} \in \mathbb Z_p[\varDelta]$, and choose $e_n \in \mathbb Z[\varDelta]$ such that $e_n \equiv e \pmod{p^n}$ for each $n \ge 1$. Put 
\[
K=\mathbb Q(\mu_{p^{\infty}})(\sqrt[p^n]{\alpha^{e_n \sigma}}\ |\ n \ge 1,\, \sigma \in G\,) \ \subset L_{\{p,\ell\}}(\mathbb Q(\mu_{p^{\infty}}))
\]
which is a Kummer extension of $\mathbb Q(\mu_{p^{\infty}})$. 

\begin{lemma}\label{K}
$K/\mathbb Q_{\infty}$ is abelian. 
\end{lemma}

\begin{proof}
Let $\delta \in \mathrm{Gal}(K/\mathbb Q_{\infty})$ be an element such that $\delta|_{\mathbb Q(\mu_{p^{\infty}})}$ generates $\varDelta$. Let $\omega_n(\delta)$ be an integer satisfying $\omega_n(\delta) \equiv \omega(\delta) \pmod{p^n}$. Since $\delta e_n \equiv \delta e = \omega(\delta) e \equiv \omega_n(\delta) e_n \pmod{p^n}$, there exists some $\beta_n \in \mathbb Q(\mu_{p^{\infty}})$ such that $(\sqrt[p^n]{\alpha^{e_n \sigma}})^{\delta} = (\sqrt[p^n]{\alpha^{e_n \sigma}})^{\omega_n(\delta)} \beta_n$. Let $g \in \mathrm{Gal}(K/\mathbb Q(\mu_{p^{\infty}}))$ be an arbitrary element. Then $(\sqrt[p^n]{\alpha^{e_n \sigma}})^{g} = \zeta_{p^n}^{z_n} \sqrt[p^n]{\alpha^{e_n \sigma}}$ with some $z_n \in \mathbb Z$ for each $n \ge 1$. Since 
$(\sqrt[p^n]{\alpha^{e_n \sigma}})^{g \delta}
= \zeta_{p^n}^{\omega_n(\delta) z_n} (\sqrt[p^n]{\alpha^{e_n \sigma}})^{\omega_n(\delta)} \beta_n
= (\sqrt[p^n]{\alpha^{e_n \sigma}})^{\delta g}$, 
we have $g \delta=\delta g$. This implies that $K/\mathbb Q_{\infty}$ is abelian. 
\end{proof}

The $p$-adic completion of $K$ is the Kummer extension 
\[
K_p=\mathbb Q_p(\mu_{p^{\infty}})(\sqrt[p^n]{\alpha^{e \sigma}}\ |\ n \ge 1,\, \sigma \in G\,) 
\]
over $\mathbb Q_p(\mu_{p^{\infty}})$. Then we obtain the following key lemma. 

\begin{lemma}\label{K_p}
$\mathrm{Gal}(K_p/\mathbb Q_p(\mu_{p^{\infty}})) \simeq {\mathbb Z_p}^{\!\oplus P_{\ell}}$ if $\ell \equiv 1 \pmod{p}$. 
\end{lemma}

\begin{proof}
Let $U=\{\, u \in \mathbb Q_p(\zeta_{qp^{N_{\ell}+1}})\ |\ v_p(u-1)>0\,\}$ be the group of principal local units of $\mathbb Q_p(\zeta_{qp^{N_{\ell}+1}})$, which is multiplicatively a $\mathbb Z_p[[\mathrm{Gal}(\mathbb Q_p(\mu_{p^{\infty}})/\mathbb Q_p)]]$-module. Since $F \subset \mathbb Q(\zeta_{qp^{N_{\ell}+1}})$, we have $\alpha \in U$. 

First, we prove the linearly independence of Kummer generators. Suppose that $\sum_{\sigma \in G} z_{\sigma} \log_p \alpha^{\sigma}=0$ with $z_{\sigma} \in \mathbb Z$, where $\log_p$ denotes the $p$-adic logarithm function. Since $\prod_{\sigma \in G} \alpha^{z_{\sigma}\sigma}$ is a root of unity, we have $\prod_{\sigma \in G} \mathfrak{L}^{(p-1)h z_{\sigma}\sigma}=O_F$ and so $z_{\sigma}=0$ for all $\sigma \in G$. Therefore $\log_p \alpha^{\sigma}$ $(\sigma \in G)$ are linearly independent over $\mathbb Q$. By the Brumer's $p$-adic version of Baker's theorem (\cite{Was} Theorem 5.29), $\log_p \alpha^{\sigma}$ $(\sigma \in G)$ are linearly independent over $\overline{\mathbb Q}\,(\subset \mathbb C_p)$. For any character $\chi : G \rightarrow \overline{\mathbb Q} \subset \mathbb C_p$, we have $\sum_{\sigma \in G} \chi(\sigma) \log_p \alpha^{\sigma} \neq 0$. Since 
\[
\det(\log_p \alpha^{\sigma \tau^{-1}})_{\sigma, \tau \in G} = \prod_{\chi \in \widehat{G}} \sum_{\sigma \in G} \chi(\sigma) \log_p \alpha^{\sigma} \neq 0
\]
by Lemma 5.26 of \cite{Was}, then $\alpha^{\sigma}$ ($\sigma \in G)$ are linearly independent in $U$ over $\mathbb Z_p$, and hence $U \supset \alpha^{\mathbb Z_p[G]} \simeq \mathbb Z_p[G]$. 
Since $\ell \equiv 1 \pmod{p}$, the restriction mapping $\varDelta \rightarrow G$ is injective, and hence there is an exact sequence $1 \rightarrow \varDelta \rightarrow G \rightarrow \varGamma/\varGamma^{P_{\ell}} \rightarrow 1$. Then we have $\overline{V}=\alpha^{e \mathbb Z_p[G]}=\alpha^{e \mathbb Z_p[\varGamma/\varGamma^{P_{\ell}}]} \simeq {\mathbb Z_p}^{\!\oplus P_{\ell}}$ as a $\mathbb Z_p$-module, which is the closure of $V=\alpha^{e \mathbb Z[\varGamma/\varGamma^{P_{\ell}}]} \simeq \mathbb Z^{\oplus P_{\ell}}$ in $U$. 

Now, we show the existence of $b$ such that $U^{p^{n+b}} \cap \overline{V} \subset \overline{V}^{p^n}$ for all $n \ge 1$. For sufficiently large $c$, $U^{p^c}$ is a free $\mathbb Z_p$-module of rank $d>P_{\ell}$. Let $u_1, \cdots, u_d$ be the basis of $U^{p^c}$ as a $\mathbb Z_p$-module. For each $\sigma \in G/\varDelta \simeq \varGamma/\varGamma^{P_{\ell}}$, $(\alpha^{e\sigma})^{p^c}=\prod_{i=1}^{d} {u_i}^{a_{\sigma,i}}$ with some $a_{\sigma,i} \in \mathbb Z_p$. Then the $P_{\ell} \times d$ matrix $A=(a_{\sigma,i})_{\sigma,i}$ has rank $P_{\ell}$, and hence there exists some $B \in GL_d(\mathbb Q_p)$ such that $AB=(E\ O)$, where $E$ is the $P_{\ell} \times P_{\ell}$ unit matrix, and $O$ is the $P_{\ell} \times (d-P_{\ell})$ zero matrix. There exists sufficiently large $b$ such that all components of $p^b B$ are $p$-adic integers. Suppose $x = \prod_{\sigma \in G/\varDelta} (\alpha^{e\sigma})^{z_{\sigma}} \in U^{p^{n+b}} \cap \overline{V}$. Then $(U^{p^c})^{p^{n+b}} \ni x^{p^c} = \prod_{\sigma \in G/\varDelta} ((\alpha^{e\sigma})^{p^c})^{z_{\sigma}} =\prod_{\sigma \in G/\varDelta} \prod_{i=1}^{d} {u_i}^{a_{\sigma,i} z_{\sigma}}$ and hence $(\cdots \,z_{\sigma} \cdots)A \equiv (\cdots \,0 \cdots) \pmod{p^{n+b}}$. Since $(\cdots \,p^b z_{\sigma} \cdots)AB = (\cdots \,z_{\sigma} \cdots)A p^b B \equiv (\cdots \,0 \cdots) \pmod{p^{n+b}}$, we have $z_{\sigma} \equiv 0 \pmod{p^n}$, i.e., $x \in \overline{V}^{p^n}$. Therefore $U^{p^{n+b}} \cap \overline{V} \subset \overline{V}^{p^n}$ where $b$ does not depend on $n$. 

The inclusions $V \subset \mathbb Q_p(\zeta_{qp^{N_{\ell}+1}})^{\times} \subset \mathbb Q_p(\mu_{p^{\infty}})^{\times}$ induce the homomorphisms 
\[
\iota^{\prime} \circ \iota : V \otimes_{\mathbb Z} \mathbb Q_p/\mathbb Z_p \stackrel{\iota}{\rightarrow} \mathbb Q_p(\zeta_{qp^{N_{\ell}+1}})^{\times} \otimes_{\mathbb Z} \mathbb Q_p/\mathbb Z_p \stackrel{\iota^{\prime}}{\rightarrow} \mathbb Q_p(\mu_{p^{\infty}})^{\times} \otimes_{\mathbb Z} \mathbb Q_p/\mathbb Z_p. 
\]
Suppose $\mathrm{Ker}\,\iota^{\prime} \neq 0$. Then there is some $x \otimes p^{-1} \in \mathrm{Ker}\,\iota^{\prime}$ such that $x \not\in {\mathbb Q_p(\zeta_{qp^{N_{\ell}+1}})^{\times}}^p$. Since $\mathbb Q_p(\zeta_{qp^{N_{\ell}+1}},\sqrt[p]{x})=\mathbb Q_p(\zeta_{qp^{N_{\ell}+2}})$, Kummer theory yields that $x=\zeta_{qp^{N_{\ell}+1}}^i y^p$ with some $i$ and $y \in \mathbb Q_p(\zeta_{qp^{N_{\ell}+1}})^{\times}$, and hence $x \otimes p^{-1}=0$. This is a contradiction. Hence $\iota^{\prime}$ is injective. Suppose $x \otimes p^{-n} \in \mathrm{Ker}\,\iota$ for $x \in V \subset \overline{V}$ and $n \ge 1$. Then there is some $u \in \mathbb Q_p(\zeta_{qp^{N_{\ell}+1}})^{\times}$ such that $x=u^{p^n}$. Since $0< v_p(u^{p^n}-1)=\sum_{i=0}^{p^n-1} v_p(u-\zeta_{p^n}^i)$, we have $v_p(u-\zeta_{p^n}^j)>0$ for some $j$. Then $v_p(u-1)=v_p(u-\zeta_{p^n}^j+\zeta_{p^n}^j-1)>0$ and hence $u \in U$. Since $x=u^{p^n} \in U^{p^n} \cap \overline{V}$, we have $x^{p^b} \in \overline{V}^{p^n}$. Since a $\mathbb Z$-basis of $V$ is a $\mathbb Z_p$-basis of $\overline{V}$, $x^{p^b} \in V^{p^n}$ and so $x^{p^b}=y^{p^n}$ with some $y \in V$. Then $p^b(x \otimes p^{-n})=x^{p^b} \otimes p^{-n}=y^{p^n} \otimes p^{-n} = y \otimes 0$ and hence $p^b \mathrm{Ker}\,\iota =0$. Since $V \otimes_{\mathbb Z} \mathbb Q_p/\mathbb Z_p \simeq (\mathbb Q_p/\mathbb Z_p)^{\oplus P_{\ell}}$, $\mathrm{Ker}\,\iota=\mathrm{Ker}(\iota^{\prime} \circ \iota)$ is finite and therefore $\mathrm{Im}(\iota^{\prime} \circ \iota) \simeq \mathrm{Im}\,\iota \simeq (\mathbb Q_p/\mathbb Z_p)^{\oplus P_{\ell}}$. Since the Kummer pairing 
\[
\mathrm{Gal}(K_p/\mathbb Q_p(\mu_{p^{\infty}})) \times \mathrm{Im}(\iota^{\prime} \circ \iota) \rightarrow \mu_{p^{\infty}}\ :\ (g, \iota^{\prime} \circ \iota(x \otimes p^{-n})) \mapsto \frac{(\sqrt[p^n]{x})^g}{\sqrt[p^n]{x}}
\]
is nondegenetate, we have $\mathrm{Gal}(K_p/\mathbb Q_p(\mu_{p^{\infty}})) \simeq \mathrm{Hom}(\mathrm{Im}\,\iota, \mu_{p^{\infty}}) \simeq {\mathbb Z_p}^{\!\oplus P_{\ell}}$ as $\mathbb Z_p$-modules. 
\end{proof}

Now we shall prove Theorem \ref{keythm}. 
\\[-3mm]%

\noindent{\it Proof of Theorem \ref{keythm}.}
Note that $X_{\emptyset}(\mathbb Q_{\infty})=0$. If $\ell \not\equiv 1 \pmod{p}$, $X_{\{\ell\}}(\mathbb Q_{\infty})=0$ by Lemma \ref{limO} and the sequence (\ref{ray}). In the following, we assume that $\ell \equiv 1 \pmod{p}$. 

Let $D_p$ (resp.\ $I_p$) be the decomposition (resp.\ inertia) subgroup of $X_{\{p,\ell\}}(\mathbb Q_{\infty})$ at the prime lying over $p$. Let $\mathbb Q_{\infty,p}$ be the $p$-adic completion of $\mathbb Q_{\infty}$. Then $D_p/I_p$ is identified with the Galois group of an unramified extension $L$ over $\mathbb Q_{\infty,p}$. Since $L=L^{\prime} \mathbb Q_{\infty,p}$ for a certain unramified extension $L^{\prime}/\mathbb Q_p$, $\varGamma=\mathrm{Gal}(\mathbb Q_{\infty}/\mathbb Q) \simeq \mathrm{Gal}(\mathbb Q_{\infty,p}/\mathbb Q_p)$ acts on $D_p/I_p$ trivially. In particular, $T=\gamma-1$ annihilates $D_p/I_p$. Since $X_{\emptyset}(\mathbb Q_{\infty})=0$, the sequence (\ref{ray}) induces the surjective homomorphism $\varprojlim (O_{\mathbb B_n}/\ell)^{\times} \otimes_{\mathbb Z} \mathbb Z_p \rightarrow X_{\{\ell\}}(\mathbb Q_{\infty})$ of $\varLambda$-modules. By Lemma \ref{limO}, $f_{\ell}(T)$ annihilates $X_{\{\ell\}}(\mathbb Q_{\infty})$. Since $D_p/I_p \subset X_{\{p,\ell\}}(\mathbb Q_{\infty})/I_p \simeq X_{\{\ell\}}(\mathbb Q_{\infty})$, we have $f_{\ell}(0)(D_p/I_p)=0$ and hence $D_p/I_p$ is finite. 

Assume that $p \neq 2$. Then $K \subset \mathbb Q(\mu_{p^{\infty}}) L_{\{p,\ell\}}(\mathbb Q_{\infty})$ by Lemma \ref{K}, and hence we obtain the following commutative diagrams. 
\[
\begin{array}{ccccc}
X_{\{p,\ell\}}(\mathbb Q_{\infty}) &\twoheadrightarrow& \mathrm{Gal}(K \cap L_{\{p,\ell\}}(\mathbb Q_{\infty})/\mathbb Q_{\infty}) &\simeq& \mathrm{Gal}(K/\mathbb Q(\mu_{p^{\infty}})) \\
\rotatebox[origin=c]{90}{$\subset$} && \rotatebox[origin=c]{90}{$\subset$} && \rotatebox[origin=c]{90}{$\hookrightarrow$} \\
D_p &\twoheadrightarrow& D_p|_{K \cap L_{\{p,\ell\}}(\mathbb Q_{\infty})} &\simeq& \mathrm{Gal}(K_p/\mathbb Q_p(\mu_{p^{\infty}}))
\end{array}
\]
By Lemmas \ref{Xpell} and \ref{K_p}, we have $\mathrm{rank}_{\mathbb Z_p} X_{\{p,\ell\}}(\mathbb Q_{\infty}) = \mathrm{rank}_{\mathbb Z_p} D_p$. Therefore $X_{\{p,\ell\}}(\mathbb Q_{\infty})/D_p$ is finite. 

Assume that $p=2$. Then $K \subset L_{\{2,\ell,\infty\}}(\mathbb Q_{\infty})$ by Lemma \ref{K}. Let $D_2^{\prime}$ be the decomposition subgroup of $X_{\{2,\ell,\infty\}}(\mathbb Q_{\infty})$ at the prime lying over $2$. Then we obtain the following commutative diagrams. 
\[
\begin{array}{ccccccc}
X_{\{2,\ell\}}(\mathbb Q_{\infty}) &\stackrel{\varphi}{\twoheadleftarrow}& X_{\{2,\ell,\infty\}}(\mathbb Q_{\infty}) &\twoheadrightarrow& \mathrm{Gal}(K/\mathbb Q_{\infty}) &\supset& \mathrm{Gal}(K/\mathbb Q(\mu_{2^{\infty}})) \\
\rotatebox[origin=c]{90}{$\subset$} && \rotatebox[origin=c]{90}{$\subset$} && \rotatebox[origin=c]{90}{$\subset$} && \rotatebox[origin=c]{90}{$\hookrightarrow$} \\
D_2 &\twoheadleftarrow& D_2^{\prime} &\twoheadrightarrow& D_2^{\prime}|_K &\hookleftarrow& \mathrm{Gal}(K_2/\mathbb Q_2(\mu_{2^{\infty}}))
\end{array}
\]
Since $\mathrm{Ker}\,\varphi$ is generated by inertia subgroups of infinite primes, $2\mathrm{Ker}\,\varphi=0$. Therefore Lemmas \ref{Xpell} and \ref{K_p} yield that $X_{\{2,\ell\}}(\mathbb Q_{\infty})/D_2$ is finite. 

Since $X_{\{\ell\}}(\mathbb Q_{\infty})/(D_p/I_p) \simeq X_{\{p,\ell\}}(\mathbb Q_{\infty})/D_p$ and $D_p/I_p$ are finite, $X_{\{\ell\}}(\mathbb Q_{\infty})$ is also finite. Thus we complete the proof of Theorem \ref{keythm}. 
\hfill$\square$
\\[-3mm]%

The proofs of Theorems \ref{mainthm} and \ref{imagquad} are also completed by Lemma \ref{reduce}. 

%------------------------------------------------------------------------------
\section{Application to finite fields}

As a consequence of Theorem \ref{keythm}, we obtain the following result on finite fields. We denote by $\mathbb F_{\ell^n}$ the finite field of cardinality $\ell^n$ and characteristic $\ell$, and by $\mathbb F_{\ell^n}^{\times}$ the multiplicative group. 

\begin{proposition}\label{finitefield}
Let $p$ and $\ell$ be odd prime numbers such that $\ell \equiv 1 \pmod{p^{N+1}}$ and $\ell \not\equiv 1 \pmod{p^{N+2}}$, where $N=N_{\ell} \ge 0$. If $n \ge N$ is sufficiently large, there exists an element $z \in \mathbb F_{\ell^{p^{n-N}}}^{\times}$ of order $p^{n+1}$ such that the order of $1-z \in \mathbb F_{\ell^{p^{n-N}}}^{\times}$ is divisible by $p$, i.e., 
\[
(1-z)^{\frac{\ell^{p^{n-N}}-1}{p^{n+1}}} \neq 1 . 
\]
\end{proposition}

\begin{proof}
Put $\varpi_n=N_{\mathbb Q(\zeta_{p^{n+1}})/\mathbb B_n}(1-\zeta_{p^{n+1}})$, which is a prime element of $\mathbb B_n$ lying over $p$, and put the group $E^{\prime}(\mathbb B_n)=O_{\mathbb B_n}[\frac{1}{\varpi_n}]^{\times}$ of $p$-units of $\mathbb B_n$. Let $\xi_n$ be the cyclotomic unit defined in the proof of Lemma \ref{limE}. Since $\xi_n \otimes 1=\varpi_n^{\overline{\gamma}-1}\otimes 1$ generates $E(\mathbb B_n) \otimes_{\mathbb Z} \mathbb Z_p$ as a $\varLambda$-module, $E^{\prime}(\mathbb B_n) \otimes_{\mathbb Z} \mathbb Z_p$ is a cyclic $\varLambda$-module generated by $\varpi_n \otimes 1$. By class field theory, we obtain the exact sequence 
\[
\varprojlim E^{\prime}(\mathbb B_n) \otimes_{\mathbb Z} \mathbb Z_p \rightarrow \varprojlim (O_{\mathbb B_n}/\ell)^{\times} \otimes_{\mathbb Z} \mathbb Z_p \rightarrow X_{\{\ell\}}(\mathbb Q_{\infty})/(D_p/I_p) \rightarrow 0 . 
\]
Since $X_{\{\ell\}}(\mathbb Q_{\infty})$ is finite by Theorem \ref{keythm} and $\varprojlim (O_{\mathbb B_n}/\ell)^{\times} \otimes_{\mathbb Z} \mathbb Z_p \simeq \varLambda/f_{\ell}(T)$ by Lemma \ref{limO}, then $(\varpi_n \bmod{\ell}) \otimes 1 \in (O_{\mathbb B_n}/\ell)^{\times} \otimes_{\mathbb Z} \mathbb Z_p$ is nontrivial for all sufficiently large $n \ge N$, and hence there is a prime ideal $\mathfrak{l}$ of $\mathbb B_n$ lying over $\ell$ such that $(\varpi_n \bmod{\mathfrak{l}})\otimes 1 \in (O_{\mathbb B_n}/\mathfrak{l})^{\times} \otimes_{\mathbb Z} \mathbb Z_p$ is also nontrivial. Since there is an isomorphism $O_{\mathbb B_n}/\mathfrak{l} \simeq \mathbb Z[\zeta_{p^{n+1}}]/\mathfrak{L}$ of finite fields, where $\mathfrak{L}$ is a prime ideal of $\mathbb Q(\zeta_{p^{n+1}})$ lying over $\mathfrak{l}$, then $(\varpi_n \bmod{\mathfrak{L}})\otimes 1$ is nontrivial. Therefore there is a primitive $p^{n+1}$th root $\zeta_{p^{n+1}}^i$ of unity such that $(1-\zeta_{p^{n+1}}^i \bmod{\mathfrak{L}})\otimes 1 \in (\mathbb Z[\zeta_{p^{n+1}}]/\mathfrak{L})^{\times} \otimes_{\mathbb Z} \mathbb Z_p$ is nontrivial. Let $z \in \mathbb F_{\ell^{p^{n-N}}}$ be the image of $\zeta_{p^{n+1}}^i \bmod{\mathfrak{L}}$ under the isomorphism $\mathbb Z[\zeta_{p^{n+1}}]/\mathfrak{L} \rightarrow \mathbb F_{\ell^{p^{n-N}}}$. Then the order of $z \in \mathbb F_{\ell^{p^{n-N}}}^{\times}$ is $p^{n+1}$, and the order of $1-z \in \mathbb F_{\ell^{p^{n-N}}}^{\times}$ is not prime to $p$. 
\end{proof}

\begin{remark}{\em 
If $N=0$, the finiteness of $X_{\{\ell\}}(\mathbb Q_{\infty})$ is equivalent to the nontriviality of $(\varpi_n \bmod{\mathfrak{L}})\otimes 1$. 
%In this case, if we assume the linear independence of $\{ (1-\zeta_{p^{n+1}}^{\delta} \bmod{\mathfrak{L}})\otimes 1 \neq 1 \otimes 0\,|\,\delta \in \varDelta \}$, the claim of Proposition \ref{finitefield} implies the finiteness of $X_{\{\ell\}}(\mathbb Q_{\infty})$. 
It seems difficult even to prove Proposition \ref{finitefield} by treating only finite fields. 
}\end{remark}

%------------------------------------------------------------------------------
\section{Order of finite Iwasawa modules}

In the case that $p \neq 2$, $X_S(\mathbb Q_{\infty})$ is finite if and only if $|S^{\prime}| \le 1$. In this section, we consider the order of $X_{\{\ell\}}(\mathbb Q_{\infty})$ in a special case. 

Put $p=3$, and let $\ell$ be a prime number such that $\ell \equiv 1 \pmod{3}$ and $\ell \not\equiv 1 \pmod{9}$. Let $\mathfrak{p}=(1-\zeta_3)$ (resp.\ $\mathfrak{L}$) be a prime ideal of $\mathbb Q(\zeta_3)$ lying over $p$ (resp.\ $\ell$). Let $\overline{\mathfrak{p}}$ be the maximal ideal of the valuation ring of the completion $\mathbb Q_p(\zeta_p)$ at $\mathfrak{p}$. Since the class number of $\mathbb Q(\zeta_3)$ is 1 and $E(\mathbb Q(\zeta_3))=\langle \zeta_6 \rangle \rightarrow (O_{\mathbb Q(\zeta_3)}/\mathfrak{p}^2)^{\times}$ is surjective, we can choose a generator $\alpha$ of $\mathfrak{L}=\alpha \mathbb Z[\zeta_3]$ such that $\alpha \equiv 1 \pmod{\mathfrak{p}^2}$. 

\begin{proposition}
Put $s=\mathrm{Tr}_{\mathbb Q(\zeta_3)/\mathbb Q} (\alpha)$ the trace of $\alpha$. Then $m=v_3(4 \ell - s^2)$ is odd, and $3 \leq |X_{\{\ell\}}(\mathbb Q_{\infty})|=3^{\frac{m-1}{2}} \le \sqrt{3}^{\frac{\log(4\ell)}{\log 3}-1}$. 
\end{proposition}

\begin{proof}
By Lemma \ref{Xpell}, $X_{\{p,\ell\}}(\mathbb Q_{\infty}) \simeq \mathbb Z_p$ and hence $X_{\{\ell\}}(\mathbb Q_{\infty})$ is a finite cyclic $p$-group. Since $A_{\{\ell\}}(\mathbb Q) \neq 0$, we have $p \le |X_{\{\ell\}}(\mathbb Q_{\infty})|$. Put $\beta=\alpha^{1-J} \not\in {\mathbb Q(\mu_{p^{\infty}})^{\times}}^p$ where $J$ is the complex conjugation. Since $p=3$, $K=\mathbb Q(\mu_{p^{\infty}})(\sqrt[p^n]{\beta}\ |\ n \ge 1\,)$ (resp.\ $K_p=K \mathbb Q_p$) coincides with the Kummer extension over $\mathbb Q(\mu_{p^{\infty}})$ (resp.\ $\mathbb Q_p(\mu_{p^{\infty}})$) defined in Section 3. Then $K=L_{\{p,\ell\}}(\mathbb Q_{\infty}) \mathbb Q(\mu_{p^{\infty}})$ by Lemma \ref{K}. Let $\overline{D}_p$ (resp.\ $\overline{I}_p$) be the decomposition (resp.\ inertia) subgroup of $\mathrm{Gal}(K/\mathbb Q(\mu_{p^{\infty}}))$ at the prime lying over $p$. Then $\overline{D}_p \simeq \mathrm{Gal}(K_p/\mathbb Q_p(\mu_{p^{\infty}}))$ and $|X_{\{\ell\}}(\mathbb Q_{\infty})|=|\mathrm{Gal}(K/\mathbb Q(\mu_{p^{\infty}}))/\overline{I}_p|$. 

Put $r=\min\{\,0 \le n \in \mathbb Z\ |\ \beta \in {\mathbb Q_p(\zeta_p)^{\times}}^{p^n}\,\}$. Then there exists some $x \in \mathbb Q_p(\zeta_p)^{\times}$ such that $\beta = x^{p^r}$ and $x \not\in {\mathbb Q_p(\zeta_p)^{\times}}^p$. Put $U^{(n)}=\{\, u \in \mathbb Q_p(\zeta_p)\ |\ u \equiv 1 \bmod{(1-\zeta_p)^n} \,\}$. Then $U^{(1)}=\langle \zeta_p \rangle \oplus U^{(2)}$ and $\beta \in U^{(2)} \simeq {\mathbb Z_p}^{\!\oplus 2}$. Since $0 < v_p(\beta-1)=\sum_{j=0}^{p^r-1} v_p((x-1)+(1-\zeta_{p^r}^j))$, we have $v_p(x-1)>0$, i.e., $x \in U^{(1)}$. Since ${U^{(1)}}^p={U^{(2)}}^p$, we may assume that $x \in U^{(2)}$, and hence $r=\min\{\,n \ge 0 \ |\ \beta \in {U^{(2)}}^{p^n}\,\}$. 

Suppose that $x \in {\mathbb Q_p(\mu_{p^{\infty}})^{\times}}^p$. Then $\mathbb Q_p(\zeta_p, \sqrt[p]{x})=\mathbb Q_p(\zeta_{p^2})$ and hence $x=\zeta_p^i y^p$ with some $i$ and $y \in \mathbb Q_p(\zeta_p)^{\times}$. Since $0<v_p(\zeta_p^{-i}x-1)=v_p(y^p-1)=\sum_{j=0}^{p-1} v_p(\zeta_p^j y-1)$, we have $y^p \in U^{(2)}$. Hence $\zeta_p^i =xy^{-p} \in U^{(2)}$, i.e., $\zeta_p^i=1$. This is a contradiction. Therefore $x \not\in {\mathbb Q_p(\mu_{p^{\infty}})^{\times}}^p$. This implies that $|\mathrm{Gal}(K/\mathbb Q(\mu_{p^{\infty}}))/\overline{D}_p|=p^r$. 

As in the proof of Theorem \ref{keythm}, let $D_p$ (resp.\ $I_p$) be the decomposition (resp.\ inertia) subgroup of $X_{\{p,\ell\}}(\mathbb Q_{\infty})$ at the prime lying over $p$. Then $\overline{D}_p/\overline{I}_p \simeq D_p/I_p$, and we have seen that $p(D_p/I_p)=f_{\ell}(0)(D_p/I_p)=0$. Since $\mathrm{Gal}(K_p/\mathbb Q_p(\zeta_p)) \simeq \mathbb Z_p \rtimes \mathbb Z_p$ is not a cyclic pro-$p$ group, the topological commutator subgroup of $\mathrm{Gal}(K_p/\mathbb Q_p(\zeta_p))$ is contained in the non-trivial maximal subgroup of $\mathrm{Gal}(K_p/\mathbb Q_p(\mu_{p^{\infty}})) \simeq \mathbb Z_p$. Hence $\mathbb Q_p(\mu_{p^{\infty}})(\sqrt[p]{x})/\mathbb Q_p(\zeta_p)$ is abelian. By Lemma \ref{K}, $\mathbb Q_p(\mu_{p^{\infty}})(\sqrt[p]{x})/\mathbb Q_p$ is also abelian. The $p$-adic version of Kronecker-Weber theorem yields that $\mathbb Q_p(\mu_{p^{\infty}})(\sqrt[p]{x})/\mathbb Q_p(\mu_{p^{\infty}})$ is a nontrivial unramified extension. Hence $|\overline{D}_p/\overline{I}_p|=p$ and therefore $|X_{\{\ell\}}(\mathbb Q_{\infty})|=p^{r+1}$. 

Since $\log_p({U^{(2)}}^{p^n})=p^n \overline{\mathfrak{p}}^2=\overline{\mathfrak{p}}^{2n+2}=\log_p(U^{(2n+2)})$ and $\log_p : U^{(2)} \rightarrow \overline{\mathfrak{p}}^2$ is an isomorphism, we have ${U^{(2)}}^{p^n}=U^{(2n+2)}$ for any $n \ge 0$. Since $\beta-1=\frac{\alpha-\alpha^J}{\alpha^J}$ and $N_{\mathbb Q(\zeta_3)/\mathbb Q} (\alpha-\alpha^J)=4 \ell-s^2$, we have $r=\max\{\,n \ge 0 \ |\ 4 \ell-s^2 \equiv 0 \bmod{p^{2n+2}}\,\}$. Since 
\[
\alpha=\frac{s \pm \sqrt{s^2-4 \ell}}{2} \in \mathbb Q(\sqrt{-3}) 
\]
and $\alpha \not\in \mathbb Q$, then $4 \ell-s^2>0$ and $m=v_p(4 \ell-s^2)$ is odd. Hence $m=2r+3$, i.e, $|X_{\{\ell\}}(\mathbb Q_{\infty})|=p^{\frac{m-1}{2}}$. Since $3^m = p^m \le 4 \ell -s^2 < 4 \ell$, then $m \le \frac{\log 4\ell}{\log 3}$. 
\end{proof}

%In the case that $p=2$, we obtain the following. 

%\begin{proposition}
%$X_{\{\ell\}}(\mathbb Q_{\infty})=0$ if $p=2$ and $\ell \equiv -1\ (\mathrm{mod}\,4)$. 
%\end{proposition}

%\begin{proof}
%Since $X_{\{\ell\}}(\mathbb Q_{\infty})/TX_{\{\ell\}}(\mathbb Q_{\infty}) \simeq A_{\{\ell\}}(\mathbb Q)=0$, Nakayama's lemma yields that $X_{\{\ell\}}(\mathbb Q_{\infty})=0$. 
%\end{proof}

%%%%%%%%%%%%%%%%%%%%%%%%%%%%%%%%%%%%%%%%%%%%%%%%%%%%%%%%%%%%%%%%%%%%%%%%%%%%%%%
\vspace*{15pt}
\noindent\textbf{Acknowledgement.} 
This work was supported by Grant-in-Aid for Young Scientists (B) (22740010) and Grant-in-Aid for Scientific Research (C) (21540030). 

%%%%%%%%%%%%%%%%%%%%%%%%%%%%%%%%%%%%%%%%%%%%%%%%%%%%%%%%%%%%%%%%%%%%%%%%%%%%%%%


\begin{reference}
%\setlength{\itemsep}{-3pt}

\bibitem{Brumer} A. Brumer, 
\textit{On the units of algebraic number fields}, 
Mathematika \textbf{14} (1967), 121--124. 

\bibitem{FW79} B. Ferrero and L. C. Washington, 
\textit{The Iwasawa invariant $\mu_p$ vanishes for abelian number fields}, 
Ann.\ of Math.\ \textbf{109} (1979), no.\ 2, 377--395.

\bibitem{Fer80} B. Ferrero, 
\textit{The cyclotomic $\mathbb Z_2$-extension of imaginary quadratic fields}, 
Amer.\ J. Math.\ \textbf{102} (1980), no.\ 3, 447--459.

\bibitem{FOO06} S. Fujii, Y. Ohgi and M. Ozaki, 
\textit{Construction of $\mathbb Z_p$-extensions with prescribed Iwasawa $\lambda$-invariants}, 
J.\ Number Theory \textbf{118} (2006), no.\ 2, 200--207. 

%%%%%\bibitem{Fuk94} T. Fukuda, 
%%%%%\textit{Remarks on $\mathbb Z_p$-extensions of number fields}, 
%%%%%Proc.\ Japan Acad.\ Ser.\ A \textbf{70} (1994), 264--266.

\bibitem{Gre76} R. Greenberg, 
\textit{On the Iwasawa invariants of totally real number fields}, 
Amer.\ J.\ Math.\ \textbf{98} (1976), no.\ 1, 263--284. 

\bibitem{IS96} H. Ichimura and H. Sumida, 
\textit{On the Iwasawa invariants of certain real abelian fields II}, 
Inter.\ J.\ Math.\ \textbf{7} (1996), no.\ 6, 721--744. 

\bibitem{Iwa59} K. Iwasawa, 
\textit{On $\Gamma$-extensions of algebraic number fields}, 
Bull.\ Amer.\ Math.\ Soc.\ \textbf{65} (1959), 183--226. 

%%%%%\bibitem{Iwa73} K. Iwasawa, 
%%%%%\textit{On $\mathbb Z_l$-extensions of algebraic number fields}, 
%%%%%Ann.\ of Math.\ (2) \textbf{98} (1973), 246--326.

\bibitem{Iwa73} K. Iwasawa, 
\textit{On the $\mu $-invariants of $\mathbb Z_l$-extensions}, 
Number theory, algebraic geometry and commutative algebra, in honor of Yasuo Akizuki, 
pp.\ 1--11.\ Kinokuniya, Tokyo, 1973.

\bibitem{Iwa81} K. Iwasawa, 
\textit{Riemann-Hurwitz formula and $p$-adic Galois representations for number fields}, 
T\^ohoku Math.\ J.\ (2) \textbf{33} (1981), no.\ 2, 263--288. 

\bibitem{Kida79} Y. Kida, 
\textit{On cyclotomic $\mathbb Z_2$-extensions of imaginary quadratic fields}, 
Tohoku Math.\ J.\ (2) \textbf{31} (1979), no.\ 1, 91--96. 

\bibitem{Kida80} Y. Kida, 
\textit{$l$-extensions of CM-fields and cyclotomic invariants}, 
J.\ Number Theory \textbf{12} (1980), no.\ 4, 519--528. 

\bibitem{Kida82} Y. Kida, 
\textit{Cyclotomic $\mathbb Z_2$-extensions of $J$-fields}, 
J. Number Theory \textbf{14} (1982), no.\ 3, 340--352. 

\bibitem{Miz00} Y. Mizusawa, 
\textit{On Greenberg's conjecture on a certain real quadratic field}, 
Proc.\ Japan Acad.\ Ser.\ A Math.\ Sci.\ \textbf{76} (2000), no.\ 10, 163--164. 

\bibitem{NSW2} J. Neukirch, A. Schmidt and K. Wingberg, 
\textit{Cohomology of Number fields (2nd.\,ed.)}, GMW \textbf{323}, Springer, 2008. 

%%%%%\bibitem{OT97} M. Ozaki and H. Taya, 
%%%%%\textit{On the Iwasawa $\lambda_2$-invariants of certain families of real quadratic fields}, 
%%%%%Manuscripta Math.\ \textbf{94} (1997), no.\,4, 437--444. 

%%%%%\bibitem{Oza07} M. Ozaki, 
%%%%%\textit{Non-Abelian Iwasawa theory of $\mathbb Z_p$-extensions}, 
%%%%%J. Reine Angew.\ Math.\ \textbf{602} (2007), 59--94. 

\bibitem{Oza09} M. Ozaki, 
\textit{Construction of real abelian fields of degree $p$ with $\lambda_p = \mu_p = 0$}, 
Int.\ J.\ Open Problems Compt.\ Math.\ \textbf{2} (2009), no.\ 3, 342--351. 

\bibitem{Salle} L. Salle, 
\textit{On maximal tamely ramified pro-$2$-extensions over the cyclotomic $\mathbb Z_2$-extension of an imaginary quadratic field}, 
Osaka J.\ Math.\ \textbf{47} (2010), no.\,4, 921--942.

\bibitem{Sin80} W. Sinnott, 
\textit{On the Stickelberger ideal and the circular units of an abelian field}, 
Invent.\ Math.\ \textbf{62} (1980), 181--234.

\bibitem{Was} 
L. C. Washington, 
\textit{Introduction to Cyclotomic Fields (2nd.\,ed.)}, 
GTM \textbf{83}, Springer, 1997.

\end{reference}
\end{document}